\documentclass[11pt]{amsart}
\usepackage{amsmath,amssymb,,latexsym,esint,cite}
\usepackage{verbatim}
\usepackage[left=3.1cm,right=3.1cm,top=3.2cm,bottom=3.2cm]{geometry}

\usepackage{color,enumitem,graphicx}
\usepackage[colorlinks=true,urlcolor=blue, citecolor=red,linkcolor=blue,linktocpage,pdfpagelabels, bookmarksnumbered,bookmarksopen]{hyperref}

\usepackage[hyperpageref]{backref}
\usepackage[english]{babel}

\newcommand{\abs}[1]{\left|#1\right|}
\newcommand{\bdry}[1]{\partial #1}
\newcommand{\bgset}[1]{\big\{#1\big\}}
\newcommand{\A}{{\mathcal A}}
\newcommand{\B}{{\mathcal B}}
\newcommand{\F}{{\mathcal F}}
\newcommand{\closure}[1]{\overline{#1}}
\newcommand{\dist}[2]{\text{dist}\, (#1,#2)}
\newcommand{\eps}{\varepsilon}
\newcommand{\half}{\frac{1}{2}}
\newcommand{\id}[1][]{id_{\, #1}}
\newcommand{\incl}{\subset}
\newcommand{\M}{{\mathcal M}}
\newcommand{\N}{\mathbb N}
\newcommand{\norm}[2][]{\left\|#2\right\|_{#1}}
\renewcommand{\O}{\text{O}}
\newcommand{\PS}[1]{$(\text{PS})_{#1}$}
\newcommand{\pnorm}[2][]{\if #1'' \left|#2\right|_p \else \left|#2\right|_{#1} \fi}
\newcommand{\R}{\mathbb R}
\newcommand{\RP}{\R \text{P}}
\newcommand{\restr}[2]{\left.#1\right|_{#2}}
\newcommand{\seq}[1]{\left(#1\right)}
\newcommand{\set}[1]{\left\{#1\right\}}
\newcommand{\strictsubset}{\subset \subset}
\newcommand{\Z}{\mathbb Z}

\DeclareMathOperator{\supp}{supp}

\newenvironment{enumroman}{\begin{enumerate}

}{\end{enumerate}}

\newtheorem{lemma}{Lemma}[section]
\newtheorem{proposition}[lemma]{Proposition}
\newtheorem{theorem}[lemma]{Theorem}

\theoremstyle{definition}
\newtheorem*{notations}{Notations}

\theoremstyle{remark}

\numberwithin{equation}{section}

\title[The Brezis-Nirenberg problem for the fractional $p$-Laplacian]{The Brezis-Nirenberg problem for \\ the fractional $p$-Laplacian}

\author[S.\ Mosconi]{Sunra Mosconi}
\author[K.\ Perera]{Kanishka Perera}
\author[M.\ Squassina]{Marco Squassina}
\author[Y.\ Yang]{Yang Yang}

\address[S.\ Mosconi]{Dipartimento di Informatica
	\newline\indent
	Universit\`a degli Studi di Verona
	\newline\indent
	Verona, Italy}
\email{sunrajohannes.mosconi@univr.it}

\address[K. Perera]{Department of Mathematical Sciences
	\newline\indent
	Florida Institute of Technology
	\newline\indent
	150 W University Blvd, Melbourne, FL 32901, USA}
\email{kperera@fit.edu}

\address[M.\ Squassina]{Dipartimento di Informatica
	\newline\indent
	Universit\`a degli Studi di Verona
	\newline\indent
	Verona, Italy}
\email{marco.squassina@univr.it}

\address[Y. Yang]{School of Science
	\newline\indent
	Jiangnan University
	\newline\indent
	Wuxi, Jiangsu 214122, China}
\email{yynjnu@126.com}


\begin{document}

\begin{abstract}
We obtain nontrivial solutions to the Brezis-Nirenberg problem for the fractional $p$\nobreakdash-Laplacian operator, extending some results in the literature for the fractional Laplacian. The quasilinear case presents two serious new difficulties. First an explicit formula for a minimizer in the fractional Sobolev inequality is not available when $p \ne 2$. We get around this difficulty by working with certain asymptotic estimates for minimizers recently obtained in \cite{BrMoSq}. The second difficulty is the lack of a direct sum decomposition suitable for applying the classical linking theorem. We use an abstract linking theorem based on the cohomological index proved in \cite{YaPe2} to overcome this difficulty.
\end{abstract}

\maketitle

\begin{center}
	\begin{minipage}{12cm}
		\small
		\tableofcontents
	\end{minipage}
\end{center}

\medskip

\section{Introduction and main result}
For $1 < p < \infty$, $s \in (0,1)$, and $N > sp$, the fractional $p$-Laplacian $(- \Delta)_p^s$ is the nonlinear nonlocal operator defined on smooth functions by
\[
(- \Delta)_p^s\, u(x) = 2\, \lim_{\eps \searrow 0} \int_{B_\eps(x)^c} \frac{|u(x) - u(y)|^{p-2}\, (u(x) - u(y))}{|x - y|^{N+sp}}\, dy, \quad x \in \R^N.
\]
This definition is consistent, up to a normalization constant depending on $N$ and $s$, with the usual definition of the linear fractional Laplacian operator $(- \Delta)^s$ when $p = 2$. There is, currently, a rapidly growing literature on problems involving these nonlocal operators. In particular, fractional $p$-eigenvalue problems have been studied in Brasco et al.\! \cite{BrPaSq}, Brasco and Parini \cite{BrPa},
Franzina and Palatucci \cite{FrPa}, Iannizzotto and Squassina \cite{MR3245079}, and Lindgren and Lindqvist \cite{MR3148135}. 
Regularity of solutions was obtained in 
Brasco and Lindgren \cite{BraLin}, Di Castro et al.\! \cite{DiKuPa,MR3237774}, Iannizzotto et al.\! \cite{IaMoSq}, Kuusi et al.\! \cite{MR3339179}, and Lindgren \cite{Li}. Existence via Morse theory was investigated in Iannizzotto et al.\! \cite{IaLiPeSq}. 
This operator appears in some recent works, see \cite{AMRT,IN} as well as \cite{Ca} for the motivations, that led to its introduction.

Let $\Omega$ be a bounded domain in $\R^N$ with Lipschitz boundary. We consider the problem
\begin{equation} \label{1}
\begin{cases}
(- \Delta)_p^s\, u  = \lambda\, |u|^{p-2}\, u + |u|^{p_s^\ast - 2}\, u & \text{in $\Omega$}  \\
u  = 0  & \text{in $\R^N \setminus \Omega$},
\end{cases}
\end{equation}
where $\lambda > 0$ and $p_s^\ast = Np/(N - sp)$ is the fractional critical Sobolev exponent. Let us recall the weak formulation of problem \eqref{1}. Let
\[
[u]_{s,p} = \left(\int_{\R^{2N}} \frac{|u(x) - u(y)|^p}{|x - y|^{N+sp}}\, dx dy\right)^{1/p}
\]
be the Gagliardo seminorm of a measurable function $u : \R^N \to \R$, and let
\[
W^{s,p}(\R^N) = \set{u \in L^p(\R^N) : [u]_{s,p} < \infty}
\]
be the fractional Sobolev space endowed with the norm
\[
\norm[s,p]{u} = \big(\pnorm{u}^p + [u]_{s,p}^p\big)^{1/p},
\]
where $\pnorm{\cdot}$ is the norm in $L^p(\R^N)$. We work in the closed linear subspace
\[
W^{s,p}_0(\Omega) = \set{u \in W^{s,p}(\R^N) : u = 0 \text{ a.e.\! in } \R^N \setminus \Omega},
\]
equivalently renormed by setting $\norm{\cdot} = [\cdot]_{s,p}$, which is a uniformly convex Banach space. 
%By \cite[Theorems 6.5 \& 7.1]{MR2944369}, 
The imbedding $W^{s,p}_0(\Omega) \hookrightarrow L^r(\Omega)$ is continuous for $r \in [1,p_s^\ast]$ and compact for $r \in [1,p_s^\ast)$. A function $u \in W^{s,p}_0(\Omega)$ is a weak solution of problem \eqref{1} if
\begin{align*}
\int_{\R^{2N}} \frac{|u(x) - u(y)|^{p-2}\, (u(x) - u(y))\, (v(x) - v(y))}{|x - y|^{N+sp}}\, dx dy &= \lambda \int_\Omega |u|^{p-2}\, uv\, dx \\
&+ \int_\Omega |u|^{p_s^\ast - 2}\, uv\, dx, \quad \forall v \in W^{s,p}_0(\Omega).
\end{align*}
See \cite{IaLiPeSq} and the references therein for further
details for this framework. In the semilinear case $p = 2$ problem \eqref{1} reduces to the critical fractional Laplacian problem
\begin{equation} \label{2}
\begin{cases}
(- \Delta)^s\, u = \lambda u + |u|^{2_s^\ast - 2}\, u & \text{in $\Omega$}  \\
u= 0 & \text{in $\R^N \setminus \Omega$},
\end{cases}
\end{equation}
where $\lambda > 0$ and $2_s^\ast = 2N/(N - 2s)$. This nonlocal problem generalizes the well-known Brezis-Nirenberg problem, which has been extensively studied beginning with the seminal paper \cite{MR709644} (see, e.g., \cite{MR779872,MR831041,MR829403,MR1009077,MR987374,MR1124117,MR1083144,MR1154480,MR1306583,MR1473856,MR1441856,MR1491613,MR1695021,MR1784441,MR1961520} and references therein). Consequently, many results known in the local case $s = 1$ have been extended to problem \eqref{2}. In particular, Servadei \cite{MR3089742,MR3237009} and Servadei and Valdinoci \cite{MR3060890,MR3271254} have shown that problem \eqref{2} has a nontrivial weak solution in the following cases:
\begin{enumroman}
\item $2s < N < 4s$ and $\lambda$ is sufficiently large;
\item $N = 4s$ and $\lambda$ is not an eigenvalue of $(- \Delta)^s$ in $\Omega$;
\item $N > 4s$.
\end{enumroman}
This extends to the fractional setting some well-known results of Brezis and Nirenberg \cite{MR709644}, Capozzi et al.\! \cite{MR831041}, Zhang \cite{MR987374}, and Gazzola and Ruf \cite{MR1441856} for critical Laplacian problems.

In the present paper we consider the quasilinear case $p \ne 2$ of problem \eqref{1}. This presents us with two serious new difficulties. Let
\[
\dot{W}^{s,p}(\R^N) = \set{u \in L^{p_s^\ast}(\R^N) : [u]_{s,p} < \infty}
\]
endowed with the norm $\norm{\cdot}$, and let
\begin{equation} \label{3}
S = \inf_{u \in \dot{W}^{s,p}(\R^N) \setminus \set{0}}\, \frac{\norm{u}^p}{\pnorm[p_s^\ast]{u}^p},
\end{equation}
which is positive by the fractional Sobolev inequality. Our first major difficulty is the lack of an explicit formula for a minimizer for $S$. It has been conjectured that all minimizers are of the form $c\, U(|x - x_0|/\eps)$, where
\[
U(x) = \frac{1}{\left(1 + |x|^{p'}\right)^{(N-sp)/p}}, \quad x \in \R^N,
\]
$p' = p/(p - 1)$ is the H\"{o}lder conjugate of $p$, $c \ne 0$, $x_0 \in \R^N$, and $\eps > 0$. This has been proved in Lieb \cite{MR717827} for $p = 2$, but for $p \ne 2$ it is not even known if these functions are minimizers. We will get around this difficulty by working with certain asymptotic estimates for minimizers recently obtained in Brasco et al.\! \cite{BrMoSq}.

Our second main difficulty is that the linking arguments based on eigenspaces of $(- \Delta)^s$ used in the case $p = 2$ do {\em not} work when $p \ne 2$ since the nonlinear operator $(- \Delta)_p^s$ does not have linear eigenspaces. We will use a more general construction based on sublevel sets as in Perera and Szulkin \cite{MR2153141} (see also Perera et al.\! \cite[Proposition 3.23]{MR2640827}). Moreover, the standard sequence of variational eigenvalues of $(- \Delta)_p^s$ based on the genus does not give enough information about the structure of the sublevel sets to carry out this linking construction. Therefore we will use a different sequence of eigenvalues introduced in Iannizzotto et al.\! \cite{IaLiPeSq} that is based on the $\Z_2$-cohomological index of Fadell and Rabinowitz \cite{MR57:17677}.

Let us recall the definition of the cohomological index. Let $W$ be a Banach space and let $\A$ denote the class of symmetric subsets of $W \setminus \set{0}$. For $A \in \A$, let $\overline{A} = A/\Z_2$ be the quotient space of $A$ with each $u$ and $-u$ identified, let $f : \overline{A} \to \RP^\infty$ be the classifying map of $\overline{A}$, and let $f^\ast : H^\ast(\RP^\infty) \to H^\ast(\overline{A})$ be the induced homomorphism of the Alexander-Spanier cohomology rings. The cohomological index of $A$ is defined by
\[
i(A) = \begin{cases}
\sup \set{m \ge 1 : f^\ast(\omega^{m-1}) \ne 0}, & A \ne \emptyset\\[5pt]
0, & A = \emptyset,
\end{cases}
\]
where $\omega \in H^1(\RP^\infty)$ is the generator of the polynomial ring $H^\ast(\RP^\infty) = \Z_2[\omega]$. For example, the classifying map of the unit sphere $S^{m-1}$ in $\R^m,\, m \ge 1$ is the inclusion $\RP^{m-1} \incl \RP^\infty$, which induces isomorphisms on $H^q$ for $q \le m - 1$, so $i(S^{m-1}) = m$.

The Dirichlet spectrum of $(- \Delta)_p^s$ in $\Omega$ consists of those $\lambda \in \R$ for which the problem
\begin{equation} \label{4}
\begin{cases}
(- \Delta)_p^s\, u  = \lambda\, |u|^{p-2}\, u & \text{in $\Omega$}  \\
u  = 0 & \text{in $\R^N \setminus \Omega$}
\end{cases}
\end{equation}
has a nontrivial weak solution. Although a complete description of the spectrum is not known when $p \ne 2$, we can define an increasing and unbounded sequence of variational eigenvalues via a suitable minimax scheme. The standard scheme based on the genus does not give the index information necessary for our purposes here, so we will use the following scheme based on the cohomological index as in Iannizzotto et al.\! \cite{IaLiPeSq} (see also Perera \cite{MR1998432}). Let
\[
\Psi(u) = \frac{1}{\pnorm{u}^p}, \quad u \in \M = \bgset{u \in W^{s,p}_0(\Omega) : \norm{u} = 1}.
\]
Then eigenvalues of problem \eqref{4} coincide with critical values of $\Psi$. We use the standard notation
\[
\Psi^a = \set{u \in \M : \Psi(u) \le a}, \quad \Psi_a = \set{u \in \M : \Psi(u) \ge a}, \quad a \in \R
\]
for the sublevel sets and superlevel sets, respectively. Let $\F$ denote the class of symmetric subsets of $\M$, and set
\[
\lambda_k := \inf_{M \in \F,\; i(M) \ge k}\, \sup_{u \in M}\, \Psi(u), \quad k \in \N.
\]
Then $0 < \lambda_1 < \lambda_2 \le \lambda_3 \le \cdots \to + \infty$ is a sequence of eigenvalues of problem \eqref{4}, and
\begin{equation} \label{5}
\lambda_k < \lambda_{k+1} \implies i(\Psi^{\lambda_k}) = i(\M \setminus \Psi_{\lambda_{k+1}}) = k
\end{equation}
(see Iannizzotto et al.\! \cite[Proposition 2.4]{IaLiPeSq}). The asymptotic behavior of these eigenvalues was recently studied in Iannizzotto and Squassina \cite{MR3245079}. Making essential use of the index information in \eqref{5}, we will prove the following theorem.

\begin{theorem}[Nonlocal Brezis-Nirenberg problem] \label{Theorem 1}
Let $1 < p < \infty$, $s \in (0,1)$, $N > sp$, and $\lambda > 0$. Then problem \eqref{1} has a nontrivial weak solution in the following cases:
\begin{enumroman}
\item $N = sp^2$ and $\lambda < \lambda_1$;
\item $N > sp^2$ and $\lambda$ is not one of the eigenvalues $\lambda_k$;
\item $N^2/(N + s) > sp^2$;
\item $(N^3 + s^3 p^3)/N\, (N + s) > sp^2$ and $\bdry{\Omega} \in C^{1,1}$.
\end{enumroman}
\end{theorem}

\noindent
This theorem extends to the fractional setting some well-known results of Garc{\'{\i}}a Azorero and Peral Alonso \cite{MR912211}, Egnell \cite{MR956567}, Guedda and V{\'e}ron \cite{MR1009077}, Arioli and Gazzola \cite{MR1741848}, and Degiovanni and Lancelotti \cite{MR2514055} for critical $p$-Laplacian problems.

Weak solutions of problem \eqref{1} coincide with critical points of the $C^1$-functional
\begin{equation} \label{6}
I_\lambda(u) = \frac{1}{p} \norm{u}^p - \frac{\lambda}{p} \pnorm{u}^p - \frac{1}{p_s^\ast} \pnorm[p_s^\ast]{u}^{p_s^\ast}, \quad u \in W^{s,p}_0(\Omega).
\end{equation}
Proof of Theorem \ref{Theorem 1} will be based on the following abstract critical point theorem 
proved in Yang and Perera (cf.\ \cite[Theorem 2.2]{YaPe2}).

\begin{theorem}
	\label{Theorem 2}
Let $W$ be a Banach space, let $S = \set{u \in W : \norm{u} = 1}$ be the unit sphere in $W$, and let $\pi : W \setminus \set{0} \to S,\, u \mapsto u/\norm{u}$ be the radial projection onto $S$. Let $I$ be a $C^1$-functional on $W$ and let $A_0$ and $B_0$ be disjoint nonempty closed symmetric subsets of $S$ such that
\[
i(A_0) = i(S \setminus B_0) < \infty.
\]
Assume that there exist $R > r > 0$ and $v \in S \setminus A_0$ such that
\[
\sup I(A) \le \inf I(B), \qquad \sup I(X) < \infty,
\]
where
\begin{align*}
A &= \set{tu : u \in A_0,\, 0 \le t \le R} \cup \set{R\, \pi((1 - t)\, u + tv) : u \in A_0,\, 0 \le t \le 1},\\
B &= \set{ru : u \in B_0},   \\
X &= \set{tu : u \in A,\, \norm{u} = R,\, 0 \le t \le 1}.
\end{align*}
Let $\Gamma = \set{\gamma \in C(X,W) : \gamma(X) \text{ is closed and} \restr{\gamma}{A} = \id[A]}$, and set
\[
c := \inf_{\gamma \in \Gamma}\, \sup_{u \in \gamma(X)}\, I(u).
\]
Then
\begin{equation} \label{43}
\inf I(B) \le c \le \sup I(X),
\end{equation}
in particular, $c$ is finite. If, in addition, $I$ satisfies the {\em \PS{c}} condition, then $c$ is a critical value of $I$.
\end{theorem}

\noindent
Theorem \ref{Theorem 2} generalizes the linking theorem of Rabinowitz \cite{MR0488128}. The linking construction in its proof was also used in Perera and Szulkin \cite{MR2153141} to obtain nontrivial solutions of $p$-Laplacian problems with nonlinearities that interact with the spectrum. A similar construction based on the notion of cohomological linking was given in Degiovanni and Lancelotti \cite{MR2371112}. See also Perera et al.\! \cite[Proposition 3.23]{MR2640827}.

\vskip3pt
\noindent
The following compactness result, proved in Perera et al.\ \cite[Proposition 3.1]{PeSqYa2}, 
will be crucial for applying Theorem \ref{Theorem 2} to our functional $I_\lambda$.

\begin{proposition}
	\label{Proposition 1}
Let $1 < p < \infty$, $s \in (0,1)$, $N > sp$, and let $S$ be as in \eqref{3}. Then for any $\lambda \in \R$, $I_\lambda$ satisfies the {\em \PS{c}} condition for all $c < \dfrac{s}{N}\, S^{N/sp}$.
\end{proposition}

\begin{notations}
We use the following notations throughout the paper. For $a \in \R$ and $q > 0$, we write $a^q = |a|^{q-1}\, a$. For $1 \le q \le \infty$, $\pnorm[q]{\cdot}$ denotes the norm in $L^q(\Omega)$ and
\[
q' = \begin{cases}
\infty, & q = 1\\[5pt]
q/(q - 1), & 1 < q < \infty\\[5pt]
1, & q = \infty
\end{cases}
\]
is the H\"{o}lder conjugate of $q$.
\end{notations}

\section*{Acknowledgments} 
Project supported by NSFC-Tian Yuan Special Foundation (No. 11226116), Natural Science Foundation of Jiangsu Province of China for Young Scholars (No. BK2012109), and the China Scholarship Council (No. 201208320435).\ 
The research  was partially supported by Gruppo Nazionale per l'Analisi Matematica,
la Probabilit\`a e le loro Applicazioni (INdAM).

\section{Preliminaries}

\subsection{Minimizers for the Sobolev inequality}

We have the following proposition from Brasco et al.\! \cite{BrMoSq} regarding the minimization problem \eqref{3}.

\begin{proposition}
Let $1 < p < \infty$, $s \in (0,1)$, $N > sp$, and let $S$ be as in \eqref{3}. Then
\begin{enumroman}
\item there exists a minimizer for $S$;
\item for every minimizer $U$, there exist $x_0 \in \R^N$ and a constant sign monotone function $u : \R \to \R$ such that $U(x) = u(|x - x_0|)$;
\item for every minimizer $U$, there exists $\lambda_U > 0$ such that
\[
\int_{\R^{2N}} \frac{(U(x) - U(y))^{p-1}\, (v(x) - v(y))}{|x - y|^{N+sp}}\, dx dy = \lambda_U \int_{\R^N} U^{p_s^\ast - 1}\, v\, dx \quad \forall v \in \dot{W}^{s,p}(\R^N).
\]
\end{enumroman}
\end{proposition}

\noindent
In the following, we shall fix a radially symmetric nonnegative decreasing minimizer $U = U(r)$
for $S$. Multiplying $U$ by a positive constant if necessary, we may assume that
\begin{equation} \label{7}
(- \Delta)_p^s\, U = U^{p_s^\ast - 1}.
\end{equation}
Testing this equation with $U$ and using \eqref{3} shows that
\begin{equation} \label{8}
\norm{U}^p = \pnorm[p_s^\ast]{U}^{p_s^\ast} = S^{N/sp}.
\end{equation}
For any $\eps > 0$, the function
\begin{equation} \label{9}
U_\eps(x) = \frac{1}{\eps^{(N-sp)/p}}\; U\bigg(\frac{|x|}{\eps}\bigg)
\end{equation}
is also a minimizer for $S$ satisfying \eqref{7} and \eqref{8}, so after a rescaling we may assume that $U(0) = 1$. Henceforth, $U$ will denote such a normalized (with respect to constant multiples and rescaling) minimizer and $U_\eps$ will denote the associated family of minimizers given by \eqref{9}. In the absence of an explicit formula for $U$, we will use the following asymptotic estimates.

\begin{lemma} \label{Lemma 1}
There exist constants $c_1, c_2 > 0$ and $\theta > 1$ such that for all $r \ge 1$,
\begin{equation} \label{10}
\frac{c_1}{r^{(N-sp)/(p-1)}} \le U(r) \le \frac{c_2}{r^{(N-sp)/(p-1)}}
\end{equation}
and
\begin{equation} \label{11}
\frac{U(\theta\, r)}{U(r)} \le \half.
\end{equation}
\end{lemma}

\begin{proof}
The inequalities in \eqref{10} were proved in Brasco et al.\! \cite{BrMoSq}. They imply
\[
\frac{U(\theta\, r)}{U(r)} \le \frac{c_2}{c_1}\, \frac{1}{\theta^{(N-sp)/(p-1)}},
\]
and \eqref{11} follows for sufficiently large $\theta$.
\end{proof}

\subsection{Regularity estimates}

Weak solutions of the equation $(- \Delta)_p^s\, u = f(x)$ enjoy the natural $L^q$-estimates given in the following lemma.

\begin{lemma} \label{Lemma 2}
Let $f \in L^q(\Omega),\, 1 < q \le \infty$ and let $u \in W^{s,p}_0(\Omega)$ be a weak solution of $(- \Delta)_p^s\, u = f(x)$ in $\Omega$. Then
\begin{equation} \label{12}
\pnorm[r]{u} \le C \pnorm[q]{f}^{1/(p-1)},
\end{equation}
where
\[
r = \begin{cases}
\dfrac{N\, (p - 1)\, q}{N - spq}, & 1 < q < \dfrac{N}{sp}\\[10pt]
\infty, & \dfrac{N}{sp} < q \le \infty
\end{cases}
\]
and $C = C(N,\Omega,p,s,q) > 0$. In particular, if $f \in L^\infty(\Omega)$, then
\[
\pnorm[\infty]{u} \le C \pnorm[\infty]{f}^{1/(p-1)}.
\]
\end{lemma}

\begin{proof}
For $k > 0$, $t \in \R$, and $\alpha > 0$, set $t_k = \max \set{-k,\min \set{t,k}}$ and consider the nondecreasing function $g(t) = t_k^\alpha$. Using Brasco and Parini \cite[Lemma A.2]{BrPa} and testing the equation $(- \Delta)_p^s\, u = f(x)$ with $g(u) \in W^{s,p}_0(\Omega)$ gives
\[
\norm{G(u)}^p \le \int_{\R^{2N}} \frac{(u(x) - u(y))^{p-1}\, (g(u(x)) - g(u(y)))}{|x - y|^{N+sp}}\, dx dy = \int_\Omega f(x)\, g(u(x))\, dx,
\]
where
\[
G(t) = \int_0^t g'(\tau)^{1/p}\, d\tau = \frac{\alpha^{1/p}\, p}{\alpha + p - 1}\, t_k^{(\alpha + p - 1)/p}.
\]
Using the Sobolev inequality on the left and the H\"{o}lder inequality on the right we get
\begin{equation} \label{13}
\big|u_k^{(\alpha + p - 1)/p}\big|_{p_s^\ast}^p \le C \pnorm[q]{f} \pnorm[q']{u_k^\alpha}.
\end{equation}
If $1 < q < N/sp$, take
\[
\alpha = \frac{(p - 1)\, p_s^\ast}{pq' - p_s^\ast} = \frac{N\, (p - 1)\, (q - 1)}{N - spq} > 0,
\]
so that
\[
\frac{\alpha + p - 1}{p}\, p_s^\ast = \alpha\, q' =: r.
\]
Then $r = N\, (p - 1)\, q/(N - spq)$ and \eqref{13} gives
\[
\pnorm[r]{u_k}^{pr/p_s^\ast} \le C \pnorm[q]{f} \pnorm[r]{u_k}^{r/q'},
\]
so
\[
|u_k|_r \le C \pnorm[q]{f}^{1/(p-1)}.
\]
Letting $k \to + \infty$ gives \eqref{12} for this case.
If $N/sp < q \le \infty$, then
\begin{equation} \label{14}
\pnorm[\infty]{u} \le C \left(\pnorm[q']{u} + \pnorm[q]{f}^{1/(p-1)}\right)
\end{equation}
by Brasco and Parini \cite[Theorem 3.1]{BrPa} and the H\"{o}lder inequality. Noting that $q' < p_s^\ast$ in this case, H\"{o}lder inequality and \eqref{13} with $\alpha = 1$ give us
\[
\pnorm[q']{u_k}^p \le C \pnorm[p_s^\ast]{u_k}^p \le C \pnorm[q]{f} \pnorm[q']{u_k},
\]
so
\[
\pnorm[q']{u_k} \le C \pnorm[q]{f}^{1/(p-1)}.
\]
Letting $k \to + \infty$ and combining with \eqref{14} gives \eqref{12} for this case.
\end{proof}

We also have the following Caccioppoli-type inequality.

\begin{lemma} \label{Lemma 7}
Let $f \in L^q(\Omega),\, 1 < q \le \infty$ and let $u \in W^{s,p}_0(\Omega)$ be a weak solution of $(- \Delta)_p^s\, u = f(x)$ in $\Omega$. If $u\, |\varphi|^p \in W^{s,p}_0(\Omega)$, then
\begin{align} \label{60}
\int_{\R^{2N}} \frac{|u(x) - u(y)|^p\, |\varphi(x)|^p}{|x - y|^{N+sp}}\, dx dy &\le 2 \int_\Omega f(x)\, u(x)\, |\varphi(x)|^p\, dx \\
&+ C \int_{\R^{2N}} \frac{|u(y)|^p\, |\varphi(x) - \varphi(y)|^p}{|x - y|^{N+sp}}\, dx dy, \notag
\end{align}
where $C = C(p) > 0$.
\end{lemma}

\begin{proof}
Testing the equation $(- \Delta)_p^s\, u = f(x)$ with $u\, |\varphi|^p$ gives
\begin{align}
&  \int_\Omega f(x)\, u(x)\, |\varphi(x)|^p\, dx \notag\\[5pt]
& =  \int_{\R^{2N}} \frac{(u(x) - u(y))^{p-1}\, (u(x)\, |\varphi(x)|^p - u(y)\, |\varphi(y)|^p)}{|x - y|^{N+sp}}\, dx dy \notag\\[5pt]
& =  \int_{\R^{2N}} \frac{|u(x) - u(y)|^p\, |\varphi(x)|^p}{|x - y|^{N+sp}}\, dx dy \notag\\[5pt]
&  + \int_{\R^{2N}} \frac{(u(x) - u(y))^{p-1}\, u(y)\, (|\varphi(x)|^p - |\varphi(y)|^p)}{|x - y|^{N+sp}}\, dx dy. \label{16}
\end{align}
By the elementary inequality $||a|^p - |b|^p| \le p\, |a - b|\, (|a|^{p-1} + |b|^{p-1})$ valid for all $a, b \in \R$ and the Young's inequality,
\begin{align*}
&  - \int_{\R^{2N}} \frac{(u(x) - u(y))^{p-1}\, u(y)\, (|\varphi(x)|^p - |\varphi(y)|^p)}{|x - y|^{N+sp}}\, dx dy\\[5pt]
& \le  p \int_{\R^{2N}} \frac{|u(x) - u(y)|^{p-1}\, |u(y)|\, |\varphi(x) - \varphi(y)|\, (|\varphi(x)|^{p-1} + |\varphi(y)|^{p-1})}{|x - y|^{N+sp}}\, dx dy\\[5pt]
& \le  \frac{1}{4} \int_{\R^{2N}} \frac{|u(x) - u(y)|^p\, (|\varphi(x)|^p + |\varphi(y)|^p)}{|x - y|^{N+sp}}\, dx dy\\[5pt]
&  + C \int_{\R^{2N}} \frac{|u(y)|^p\, |\varphi(x) - \varphi(y)|^p}{|x - y|^{N+sp}}\, dx dy\\[5pt]
& =  \half \int_{\R^{2N}} \frac{|u(x) - u(y)|^p\, |\varphi(x)|^p}{|x - y|^{N+sp}}\, dx dy + C \int_{\R^{2N}} \frac{|u(y)|^p\, |\varphi(x) - \varphi(y)|^p}{|x - y|^{N+sp}}\, dx dy.
\end{align*}
Combining this with \eqref{16} gives \eqref{60}.
\end{proof}

\noindent
As a consequence of Lemmas \ref{Lemma 2} and \ref{Lemma 7}, we have the following lemma.

\begin{lemma} \label{Lemma 4}
Let $f \in L^q(\Omega),\, N/sp < q \le \infty$ and let $u \in W^{s,p}_0(\Omega)$ be a weak solution of $(- \Delta)_p^s\, u = f(x)$ in $\Omega$. Then
\begin{equation} \label{15}
\norm{u \varphi}^p \le C \pnorm[q]{f}^{p/(p-1)} \left(\pnorm[pq']{\varphi}^p + \norm{\varphi}^p\right) \quad \forall \varphi \in L^{pq'}(\Omega) \cap W^{s,p}_0(\Omega),
\end{equation}
where $C = C(N,\Omega,p,s,q) > 0$.
\end{lemma}

\begin{proof}
Setting $t_k = \max \set{-k,\min \set{t,k}}$ for $k > 0$ and $t \in \R$, noting that $u\, |\varphi_k|^p \in W^{s,p}_0(\Omega)$, and applying Lemma \ref{Lemma 7} gives
\begin{align} \label{17}
\int_{\R^{2N}} \frac{|u(x) - u(y)|^p\, |\varphi_k(x)|^p}{|x - y|^{N+sp}}\, dx dy &\le 2 \int_\Omega f(x)\, u(x)\, |\varphi_k(x)|^p\, dx \\
&+ C \int_{\R^{2N}} \frac{|u(y)|^p\, |\varphi_k(x) - \varphi_k(y)|^p}{|x - y|^{N+sp}}\, dx dy.  \notag
\end{align}
Since $N/sp < q \le \infty$,
\begin{equation} \label{18}
\pnorm[\infty]{u} \le C \pnorm[q]{f}^{1/(p-1)}
\end{equation}
by Lemma \ref{Lemma 2}. By \eqref{17}, \eqref{18}, and the H\"{o}lder inequality,
\begin{align*}
\int_{\R^{2N}} \frac{|u(x) - u(y)|^p\, |\varphi_k(x)|^p}{|x - y|^{N+sp}}\, dx dy &\le C \pnorm[q]{f}^{p/(p-1)} \left(\pnorm[pq']{\varphi_k}^p + \norm{\varphi_k}^p\right) \\
&\le C \pnorm[q]{f}^{p/(p-1)} \left(\pnorm[pq']{\varphi}^p + \norm{\varphi}^p\right),
\end{align*}
and letting $k \to + \infty$ gives
\begin{equation} \label{19}
\int_{\R^{2N}} \frac{|u(x) - u(y)|^p\, |\varphi(x)|^p}{|x - y|^{N+sp}}\, dx dy \le C \pnorm[q]{f}^{p/(p-1)} \left(\pnorm[pq']{\varphi}^p + \norm{\varphi}^p\right).
\end{equation}
Since
\begin{align*}
\int_{\R^{2N}} \frac{|u(x)\, \varphi(x) - u(y)\, \varphi(y)|^p}{|x - y|^{N+sp}}\, dx dy &\le C \left(\int_{\R^{2N}} \frac{|u(x) - u(y)|^p\, |\varphi(x)|^p}{|x - y|^{N+sp}}\, dx dy\right. \\
&+ \left.\int_{\R^{2N}} \frac{|u(y)|^p\, |\varphi(x) - \varphi(y)|^p}{|x - y|^{N+sp}}\, dx dy\right),
\end{align*}
\eqref{15} readily follows from \eqref{19} and \eqref{18}.
\end{proof}

\noindent
Now let $\theta$ be as in Lemma \ref{Lemma 1}, let $\eta \in C^\infty(\R^N,[0,1])$ be such that
\[
\eta(x) = \begin{cases}
0, & |x| \le 2 \theta\\[5pt]
1, & |x| \ge 3 \theta,
\end{cases}
\]
and let $\eta_\delta(x) = \eta\Big(\dfrac{x}{\delta}\Big)$ for $\delta > 0$.

\begin{lemma} \label{Lemma 6}
Assume that $0 \in \Omega$. Then there exists a constant $C = C(N,\Omega,p,s) > 0$ such that for any $v \in W^{s,p}_0(\Omega)$ such that $(- \Delta)_p^s\, v \in L^\infty(\Omega)$ and $\delta > 0$ such that $B_{5 \theta \delta}(0) \subset \Omega$,
\[
\norm{v \eta_\delta}^p \le \norm{v}^p + C \pnorm[\infty]{(- \Delta)_p^s\, v}^{p/(p-1)} \delta^{N-sp}.
\]
\end{lemma}

\begin{proof}
We have
\begin{align} \label{59}
\norm{v \eta_\delta}^p &\le \int_{A_1} \frac{|v(x) - v(y)|^p}{|x - y|^{N+sp}}\, dx dy + \int_{A_2} \frac{|v(x)\, \eta_\delta(x) - v(y)\, \eta_\delta(y)|^p}{|x - y|^{N+sp}}\, dx dy  \\
& + 2 \int_{A_3} \frac{|v(x)\, \eta_\delta(x) - v(y)|^p}{|x - y|^{N+sp}}\, dx dy =: I_1 + I_2 + 2 I_3, \notag
\end{align}
where 
$$
A_1 = B_{3 \theta \delta}(0)^c \times B_{3 \theta \delta}(0)^c, \quad
A_2 = B_{4 \theta \delta}(0) \times B_{4 \theta \delta}(0),\quad 
A_3 = B_{3 \theta \delta}(0) \times B_{4 \theta \delta}(0)^c. 
$$
Clearly, $I_1 \le \norm{v}^p$. To estimate $I_2$, let $\varphi \in C^\infty_0(B_{5 \theta}(0),[0,1])$ with $\varphi = \eta$ in $B_{4 \theta}(0)$ and let $\varphi_\delta(x) = \varphi(x/\delta)$. Then
\[
I_2 = \int_{A_2} \frac{|v(x)\, \varphi_\delta(x) - v(y)\, \varphi_\delta(y)|^p}{|x - y|^{N+sp}}\, dx dy \le \norm{v \varphi_\delta}^p \le C \pnorm[\infty]{(- \Delta)_p^s\, v}^{p/(p-1)} \norm{\varphi_\delta}^p
\]
by Lemma \ref{Lemma 4} applied to $\varphi_\delta$ with $q = \infty$, and $\norm{\varphi_\delta}^p = \delta^{N-sp} \norm{\varphi}^p$. Since $|x - y| \ge |y| - 3 \theta \delta \ge |y|/4$ on $A_3$,
\[
I_3 \le C \pnorm[\infty]{v}^p \int_{A_3} \frac{dx dy}{|y|^{N+sp}} \le C \pnorm[\infty]{(- \Delta)_p^s\, v}^{p/(p-1)} \delta^{N-sp}
\]
by Lemma \ref{Lemma 2}.
\end{proof}

\subsection{Auxiliary estimates}

We now construct some auxiliary functions and estimate their norms. In what follows $\theta$ is the universal constant in Lemma \ref{Lemma 1} that depends only on $N$, $p$, and $s$. We may assume without loss of generality that $0 \in \Omega$.
For $\eps, \delta > 0$, let
\[
m_{\eps,\delta} = \frac{U_\eps(\delta)}{U_\eps(\delta) - U_\eps(\theta \delta)},
\]
let
\[
g_{\eps,\delta}(t) = \begin{cases}
0, & 0 \le t \le U_\eps(\theta \delta)\\[5pt]
m_{\eps,\delta}^p\, (t - U_\eps(\theta \delta)), & U_\eps(\theta \delta) \le t \le U_\eps(\delta)\\[5pt]
t + U_\eps(\delta)\, (m_{\eps,\delta}^{p-1} - 1), & t \ge U_\eps(\delta),
\end{cases}
\]
and let
\begin{equation}
\label{defG}
G_{\eps,\delta}(t) = \int_0^t g_{\eps,\delta}'(\tau)^{1/p}\, d\tau = \begin{cases}
0, & 0 \le t \le U_\eps(\theta \delta)\\[5pt]
m_{\eps,\delta}\, (t - U_\eps(\theta \delta)), & U_\eps(\theta \delta) \le t \le U_\eps(\delta)\\[5pt]
t, & t \ge U_\eps(\delta).
\end{cases}
\end{equation}
The functions $g_{\eps,\delta}$ and $G_{\eps,\delta}$ are nondecreasing and absolutely continuous. Consider the radially symmetric nonincreasing function
\[
u_{\eps,\delta}(r) = G_{\eps,\delta}(U_\eps(r)),
\]
which satisfies
\begin{equation} \label{20}
u_{\eps,\delta}(r) = \begin{cases}
U_\eps(r), & r \le \delta\\[5pt]
0, & r \ge \theta \delta.
\end{cases}
\end{equation}
We have the following estimates for $u_{\eps,\delta}$.

\begin{lemma} \label{Lemma 3}
There exists a constant $C = C(N,p,s) > 0$ such that for any $\eps \le \delta/2$,
\begin{gather}
\label{21} \norm{u_{\eps,\delta}}^p \le S^{N/sp} + C \left(\frac{\eps}{\delta}\right)^{(N-sp)/(p-1)},\\[10pt]
\label{22} \pnorm{u_{\eps,\delta}}^p \ge \begin{cases}
\dfrac{1}{C}\; \eps^{sp}\, \log \bigg(\dfrac{\delta}{\eps}\bigg), & N = sp^2\\[10pt]
\dfrac{1}{C}\; \eps^{sp}, & N > sp^2,
\end{cases}\\[12.5pt]
\label{23} \pnorm[p_s^\ast]{u_{\eps,\delta}}^{p_s^\ast} \ge S^{N/sp} - C \left(\frac{\eps}{\delta}\right)^{N/(p-1)}.
\end{gather}
\end{lemma}

\begin{proof}
Using Brasco and Parini \cite[Lemma A.2]{BrPa} and testing the equation $(- \Delta)_p^s\, U_\eps = U_\eps^{p_s^\ast - 1}$ with $g_{\eps,\delta}(U_\eps) \in W^{s,p}_0(\Omega)$ gives
\begin{align*}
\norm{G_{\eps,\delta}(U_\eps)}^p & \le \int_{\R^{2N}} \frac{(U_\eps(x) - U_\eps(y))^{p-1}\, (g_{\eps,\delta}(U_\eps(x)) - g_{\eps,\delta}(U_\eps(y)))}{|x - y|^{N+sp}}\, dx dy  \\
&= \int_{\R^N} U_\eps(x)^{p_s^\ast - 1}\, g_{\eps,\delta}(U_\eps(x))\, dx  \\
&= \pnorm[p_s^\ast]{U_\eps}^{p_s^\ast} + \int_{\R^N} (g_{\eps,\delta}(U_\eps(x)) - U_\eps(x))\, U_\eps(x)^{p_s^\ast - 1}\, dx.
\end{align*}
We have $\pnorm[p_s^\ast]{U_\eps}^{p_s^\ast} = S^{N/sp}$ by \eqref{8},
\begin{align*}
g_{\eps,\delta}(t) - t \le U_\eps(\delta)\, m_{\eps,\delta}^{p-1} & = \frac{1}{\eps^{(N-sp)/p}}\; U\bigg(\frac{\delta}{\eps}\bigg)\! \left[1 - U\bigg(\dfrac{\theta \delta}{\eps}\bigg)\bigg/U\bigg(\dfrac{\delta}{\eps}\bigg)\right]^{-(p-1)}\\[5pt]
&\le 2^{p-1}\, c_2\, \frac{\eps^{(N-sp)/p(p-1)}}{\delta^{(N-sp)/(p-1)}}, \,\,\quad \forall t \ge 0
\end{align*}
by \eqref{10} and \eqref{11},
\[
\int_{\R^N} U_\eps(x)^{p_s^\ast - 1}\, dx = \eps^{(N-sp)/p} \int_{\R^N} U(x)^{p_s^\ast - 1}\, dx,
\]
and the last integral is finite by \eqref{10} again, so \eqref{21} follows.
Using \eqref{20},
\[
\int_{\R^N} u_{\eps,\delta}(x)^p\, dx \ge \int_{B_\delta(0)} u_{\eps,\delta}(x)^p\, dx = \int_{B_\delta(0)} U_\eps(x)^p\, dx = \eps^{sp} \int_{B_{\delta/\eps}(0)} U(x)^p\, dx,
\]
and the last integral is greater than or equal to
\[
\int_1^{\delta/\eps} U(r)^p\, r^{N-1}\, dr \ge c_1^p \int_1^{\delta/\eps} r^{-(N-sp^2)/(p-1)-1}\, dr
\]
by \eqref{10}. A direct evaluation of the integral on the right gives \eqref{22} since $\delta/\eps \ge 2$.
Using \eqref{20} again,
\begin{align*}
\int_{\R^N} u_{\eps,\delta}(x)^{p_s^\ast}\, dx &\ge \int_{B_\delta(0)} u_{\eps,\delta}(x)^{p_s^\ast}\, dx = \int_{B_\delta(0)} U_\eps(x)^{p_s^\ast}\, dx \\
&= S^{N/sp} - \int_{B_{\delta/\eps}(0)^c} U(x)^{p_s^\ast}\, dx
\end{align*}
by \eqref{8}. By \eqref{10}, the last integral is less than or equal to
\[
c_2^{p_s^\ast} \int_{\delta/\eps}^\infty r^{-N/(p-1)-1}\, dr = \frac{(p - 1)\, c_2^{p_s^\ast}}{N} \left(\frac{\eps}{\delta}\right)^{N/(p-1)},
\]
so \eqref{23} follows.
\end{proof}

\noindent
We note that Lemma \ref{Lemma 3} gives the following estimate for
\[
S_{\eps,\delta}(\lambda) := \frac{\norm{u_{\eps,\delta}}^p - \lambda \pnorm{u_{\eps,\delta}}^p}{\pnorm[p_s^\ast]{u_{\eps,\delta}}^p}:
\]
there exists a constant $C = C(N,p,s) > 0$ such that for any $\eps \le \delta/2$,
\begin{equation} \label{40}
S_{\eps,\delta}(\lambda) \le \begin{cases}
S - \dfrac{\lambda}{C}\; \eps^{sp}\, \log \bigg(\dfrac{\delta}{\eps}\bigg) + C\, \bigg(\dfrac{\eps}{\delta}\bigg)^{sp}, & N = sp^2\\[10pt]
S - \dfrac{\lambda}{C}\; \eps^{sp} + C\, \bigg(\dfrac{\eps}{\delta}\bigg)^{(N-sp)/(p-1)}, & N > sp^2.
\end{cases}
\end{equation}

\medskip

\section{Proof of the main result}

In this section we prove Theorem \ref{Theorem 1}. For $0 < \lambda < \lambda_1$, mountain pass theorem and \eqref{40} will give us a positive critical level of $I_\lambda$ below the threshold level for compactness given in Proposition \ref{Proposition 1}. For $\lambda \ge \lambda_1$, we will use the abstract linking theorem, Theorem \ref{Theorem 2}.

\subsection{Case 1: $N \ge sp^2$ and $0 < \lambda < \lambda_1$}

We have
\[
I_\lambda(u) \ge \frac{1}{p} \left(1 - \frac{\lambda}{\lambda_1}\right) \norm{u}^p - \frac{1}{p_s^\ast\, S^{p_s^\ast/p}} \norm{u}^{p_s^\ast},
\]
so the origin is a strict local minimizer of $I_\lambda$. Fix $\delta > 0$ so small that $B_{\theta \delta}(0) \strictsubset \Omega$, so that $\supp u_{\eps,\delta} \subset \Omega$ by \eqref{20}. Noting that
\[
I_\lambda(Ru_{\eps,\delta}) = \frac{R^p}{p} \left(\norm{u_{\eps,\delta}}^p - \lambda \pnorm{u_{\eps,\delta}}^p\right) - \frac{R^{p_s^\ast}}{p_s^\ast} \pnorm[p_s^\ast]{u_{\eps,\delta}}^{p_s^\ast} \to - \infty \quad \text{as } R \to + \infty,
\]
fix $R_0 > 0$ so large that $I_\lambda(R_0 u_{\eps,\delta}) < 0$. Then let
\[
\Gamma = \set{\gamma \in C([0,1],W^{s,p}_0(\Omega)) : \gamma(0) = 0,\, \gamma(1) = R_0 u_{\eps,\delta}}
\]
and set
\[
c := \inf_{\gamma \in \Gamma}\, \max_{t \in [0,1]}\, I_\lambda(\gamma(t)) > 0.
\]
Since $t \mapsto tR_0 u_{\eps,\delta}$ is a path in $\Gamma$,
\begin{equation} \label{24}
c \le \max_{t \in [0,1]}\, I_\lambda(tR_0 u_{\eps,\delta}) = \frac{s}{N} \left(\frac{\norm{u_{\eps,\delta}}^p - \lambda \pnorm{u_{\eps,\delta}}^p}{\pnorm[p_s^\ast]{u_{\eps,\delta}}^p}\right)^{N/sp} = \frac{s}{N}\, S_{\eps,\delta}(\lambda)^{N/sp}.
\end{equation}
By \eqref{40},
\[
S_{\eps,\delta}(\lambda) \le \begin{cases}
S + \left(C - \dfrac{\lambda}{C} \abs{\log \eps}\right) \eps^{sp}, & N = sp^2\\[10pt]
S - \left(\dfrac{\lambda}{C} - C\, \eps^{(N-sp^2)/(p-1)}\right) \eps^{sp}, & N > sp^2,
\end{cases}
\]
so $S_{\eps,\delta}(\lambda) < S$ if $\eps > 0$ is sufficiently small. So 
$$
c < \dfrac{s}{N}\, S^{N/sp}
$$ 
by \eqref{24}, and hence $I_\lambda$ satisfies the \PS{c} condition by Proposition \ref{Proposition 1}. Then $c$ is a critical level of $I_\lambda$ by the mountain pass theorem.

\subsection{Case 2: $N > sp^2$ and $\lambda > \lambda_1$ is not one of the eigenvalues $\lambda_k$}

We have $\lambda_k < \lambda < \lambda_{k+1}$ for some $k \in \N$, and then $i(\Psi^{\lambda_k}) = i(\M \setminus \Psi_{\lambda_{k+1}}) = k$ by \eqref{5}. In what follows
\[
\pi(u) = \frac{u}{\norm{u}}, \quad \pi_p(u) = \frac{u}{\pnorm{u}}, \quad u \in W^{s,p}_0(\Omega) \setminus \set{0}
\]
are the radial projections onto
\[
\M = \bgset{u \in W^{s,p}_0(\Omega) : \norm{u} = 1}, \quad \M_p = \bgset{u \in W^{s,p}_0(\Omega) : \pnorm{u} = 1},
\]
respectively.

\begin{proposition} \label{Proposition 2}
If $\lambda_k < \lambda_{k+1}$, then $\Psi^{\lambda_k}$ has a compact symmetric subset $E$ with $i(E) = k$ such that
\[
\pnorm[\infty]{(- \Delta)_p^s\, v} \le C \quad \forall v \in E,
\]
where $C = C(N,\Omega,p,s,k) > 0$. In particular,
\[
\pnorm[\infty]{v} \le C \quad \forall v \in E.
\]
\end{proposition}

\begin{proof}
For $w \in L^q(\Omega)$ with $q \ge \max \set{1,(p - 1)\, (p_s^\ast)'}$, the equation $(- \Delta)_p^s\, u = |w|^{p-2}\, w$ has a unique weak solution $u = \B(w) \in W^{s,p}_0(\Omega)$. By Lemma \ref{Lemma 2},
\begin{equation} \label{25}
\pnorm[\gamma(q)]{\B(w)} \le C(q) \pnorm[q]{w},
\end{equation}
where
\[
\gamma(q) = \begin{cases}
\dfrac{N\, (p - 1)\, q}{N\, (p - 1) - spq}, & \dfrac{q}{p - 1} < \dfrac{N}{sp}\\[10pt]
\infty, & \dfrac{N}{sp} < \dfrac{q}{p - 1} \le \infty.
\end{cases}
\]
For $w \in \M_p$, let $J(w) = \pi_p(u) \in \M_p$, where $u = \B(w)$. Testing $(- \Delta)_p^s\, u = |w|^{p-2}\, w$ with $u,\, w$ and using the H\"{o}lder inequality gives
\begin{align*}
\norm{u}^p &= \int_\Omega |w|^{p-2}\, wu\, dx \le \pnorm{w}^{p-1} \pnorm{u} = \pnorm{u}, \\
1 &= \pnorm{w}^p = \int_{\R^{2N}} \frac{(u(x) - u(y))^{p-1}\, (w(x) - w(y))}{|x - y|^{N+sp}}\, dx dy \le \norm{u}^{p-1} \norm{w},
\end{align*}
respectively, so
\begin{equation} \label{26}
\norm{J(w)} = \frac{\norm{u}}{\pnorm{u}} \le \frac{1}{\norm{u}^{p-1}} \le \norm{w}, \qquad \pnorm{\B(w)} = \pnorm{u} \ge \norm{u}^p \ge \frac{1}{\norm{w}^{p/(p-1)}}.
\end{equation}
Let
\[
A = \pi_p(\Psi^{\lambda_k}) = \bgset{w \in \M_p : \norm{w}^p \le \lambda_k}.
\]
Then $i(A) = i(\Psi^{\lambda_k}) = k$ by the monotonicity of the index and \eqref{5}, and $A$ is strongly compact in $L^p(\Omega)$. By \eqref{26}, $J(A) \subset A$ and
\begin{equation} \label{27}
\pnorm{\B(w)} \ge \frac{1}{\lambda_k^{1/(p-1)}} \quad \forall w \in A.
\end{equation}
For $w \in A$, if $p/(p - 1) > N/sp$, then $\gamma(p) = \infty$ and hence
\[
\pnorm[\infty]{J(w)} = \frac{\pnorm[\infty]{\B(w)}}{\pnorm{\B(w)}} \le C(p)\, \lambda_k^{1/(p-1)} \pnorm{w} = C(p)\, \lambda_k^{1/(p-1)}
\]
by \eqref{25} and \eqref{27}. Otherwise, take $\max \set{1,(p - 1)\, (p_s^\ast)'} \le q_0 < p$ and define the sequence $\seq{q_i}$ recursively by setting $q_i = \gamma(q_{i-1})$ if $q_{i-1}/(p - 1) < N/sp$, in which case
\[
q_i - q_{i-1} = \frac{sp\, q_{i-1}^2}{N\, (p - 1) - sp\, q_{i-1}} \ge \frac{sp}{N\, (p - 1) - sp} > 0.
\]
Hence $q_0$ may be chosen so that $q_{n-1}/(p - 1) < N/sp < q_n/(p - 1)$ for some $n \ge 1$. Iterating \eqref{25} and \eqref{27}, and using the H\"{o}lder inequality at the last step then gives
\begin{align} \label{28}
\pnorm[\infty]{J^n(w)} &= \frac{\pnorm[\infty]{\B(J^{n-1}(w))}}{\pnorm{\B(J^{n-1}(w))}} \le C(q_n)\, \lambda_k^{1/(p-1)} \pnorm[q_n]{J^{n-1}(w)} \le \cdots\\
&\le C(q_n) \cdots C(q_0)\, \lambda_k^{(n+1)/(p-1)} \pnorm[q_0]{w} \le C\, \lambda_k^{(n+1)/(p-1)}. \notag
\end{align}
Let $\widetilde{A} = J^{n+1}(A) \subset A$. For each $\widetilde{v} \in \widetilde{A}$, there exists $\widetilde{w} \in J^n(A) \subset A$ such that $\widetilde{v} = J(\widetilde{w}) = u/\pnorm{u}$, where $u = \B(\widetilde{w})$. Then
\[
(- \Delta)_p^s\, \widetilde{v} = \frac{(- \Delta)_p^s\, u}{\pnorm{u}^{p-1}} = \frac{|\widetilde{w}|^{p-2}\, \widetilde{w}}{\pnorm{\B(\widetilde{w})}^{p-1}},
\]
so
\begin{equation} \label{29}
\pnorm[\infty]{(- \Delta)_p^s\, \widetilde{v}} = \frac{\pnorm[\infty]{\widetilde{w}}^{p-1}}{\pnorm{\B(\widetilde{w})}^{p-1}} \le C\, \lambda_k^{n+2}
\end{equation}
by \eqref{27} and \eqref{28}.\ Since the imbedding $W^{s,p}_0(\Omega) \hookrightarrow L^p(\Omega)$ is compact and $J$ is an odd continuous map from $L^p(\Omega)$ to $W^{s,p}_0(\Omega)$, $\widetilde{A}$ is a compact set and $i(\widetilde{A}) = i(A) = k$.

Let $E = \pi(\widetilde{A})$ and note that $E$ is compact with $i(E) = i(\widetilde{A}) = k$. For each $v \in E$, there exists $\widetilde{v} \in \widetilde{A} \subset A$ such that $v = \widetilde{v}/\norm{\widetilde{v}}$. Then
\[
\Psi(v) = \frac{\norm{\widetilde{v}}^p}{\pnorm{\widetilde{v}}^p} \le \lambda_k,
\]
so $E \subset \Psi^{\lambda_k}$. Since $1 = \pnorm{\widetilde{v}} \le C \norm{\widetilde{v}}$,
\[
\pnorm[\infty]{(- \Delta)_p^s\, v} = \frac{\pnorm[\infty]{(- \Delta)_p^s\, \widetilde{v}}}{\norm{\widetilde{v}}^{p-1}} \le C\, \lambda_k^{n+2}
\]
by \eqref{29}.
\end{proof}

\noindent
For $v \in E$, let $v_\delta = v \eta_\delta$, where $\eta_\delta$ is the cut-off function in Lemma \ref{Lemma 6}, and let
\[
E_\delta = \set{\pi(v_\delta) : v \in E}.
\]

\begin{proposition} \label{Proposition 3}
There exists a constant $C = C(N,\Omega,p,s,k) > 0$ such that for all sufficiently small $\delta > 0$,
\begin{gather}
\label{30} \frac{1}{C} \le \pnorm[q]{w} \le C \quad \forall w \in E_\delta,\, 1 \le q \le \infty,\\[7.5pt]
\label{31} \sup_{w \in E_\delta}\, \Psi(w) \le \lambda_k + C \delta^{N-sp},
\end{gather}
$E_\delta \cap \Psi_{\lambda_{k+1}} = \emptyset$, $i(E_\delta) = k$, and $\supp w \subset B_{2 \theta \delta}(0)^c$ for all $w \in E_\delta$. In particular, the supports of $w$ and $\pi(u_{\eps,\delta})$ are disjoint and hence $\pi(u_{\eps,\delta}) \not\in E_\delta$.
\end{proposition}

\begin{proof}
Let $v \in E$ and let $w = \pi(v_\delta)$. We have
\[
\int_\Omega |v|^q\, dx = \int_{\Omega \setminus B_{3 \theta \delta}(0)} |v_\delta|^q\, dx + \int_{B_{3 \theta \delta}(0)} |v|^q\, dx \le \int_\Omega |v_\delta|^q\, dx + C \pnorm[\infty]{v}^q \delta^N,
\]
so
\begin{equation} \label{33}
\int_\Omega |v_\delta|^q\, dx \ge \int_\Omega |v|^q\, dx - C \delta^N
\end{equation}
by Proposition \ref{Proposition 2}. In particular, $\pnorm[1]{v_\delta} \ge \pnorm[1]{v} - C \delta^N$. On the other hand,
\[
1 = \norm{v}^p = \int_\Omega v\, (- \Delta)_p^s\, v\, dx \le \pnorm[\infty]{(- \Delta)_p^s\, v} \int_\Omega |v|\, dx \le C \pnorm[1]{v}
\]
by Proposition \ref{Proposition 2} again, so $\pnorm[1]{v_\delta} \ge 1/C - C \delta^N$. Since
\begin{equation} \label{32}
\norm{v_\delta}^p \le 1 + C \delta^{N-sp}
\end{equation}
by Lemma \ref{Lemma 6} and Proposition \ref{Proposition 2}, then
\[
\pnorm[1]{w} = \frac{\pnorm[1]{v_\delta}}{\norm{v_\delta}} \ge \frac{\dfrac{1}{C} - C \delta^N}{1 + C \delta^{(N-sp)/p}} = \frac{1}{C} + \O(\delta^{(N-sp)/p}),
\]
which together with the H\"{o}lder inequality gives the first half of \eqref{30}. By \eqref{33} with $q = p$,
\begin{equation} \label{39}
\pnorm{v_\delta}^p \ge \pnorm{v}^p - C \delta^N \ge \frac{1}{\lambda_k} - C \delta^N
\end{equation}
since $E \subset \Psi^{\lambda_k}$. So $\pnorm{v_\delta}$, and hence also $\norm{v_\delta}$, is bounded away from zero. Since $\pnorm[\infty]{v}$ is bounded by Proposition \ref{Proposition 2} and $0 \le \eta_\delta \le 1$, $\pnorm[\infty]{v_\delta}$ is bounded, so this shows that $\pnorm[\infty]{w} = \pnorm[\infty]{v_\delta}/\norm{v_\delta}$ is bounded, which gives the second half of \eqref{30}.

Combining \eqref{32} and \eqref{39} gives
\[
\Psi(w) = \frac{\norm{v_\delta}^p}{\pnorm{v_\delta}^p} \le \frac{1 + C \delta^{N-sp}}{\dfrac{1}{\lambda_k} - C \delta^N} = \lambda_k + \O(\delta^{N-sp}).
\]
Fix $\delta > 0$ so small that $\lambda_k + C \delta^{N-sp} < \lambda_{k+1}$. Then $E_\delta \subset \M \setminus \Psi_{\lambda_{k+1}}$ by \eqref{31}, and hence $i(E_\delta) \le i(\M \setminus \Psi_{\lambda_{k+1}}) = k$ by the monotonicity of the index and \eqref{5}. On the other hand, $E \to E_\delta,\, v \mapsto \pi(v_\delta)$ is an odd continuous map and hence $i(E_\delta) \ge i(E) = k$. So $i(E_\delta) = k$.

Finally, $\supp \pi(v_\delta) = \supp v_\delta \subset \supp \eta_\delta \subset B_{2 \theta \delta}(0)^c$ for all $v \in E$, and 
$$
\supp \pi(u_{\eps,\delta}) = \supp u_{\eps,\delta} \subset \closure{B_{\theta \delta}(0)},
$$ 
by virtue of \eqref{20}.
\end{proof}

We are now ready to apply Theorem \ref{Theorem 2} to obtain a nontrivial critical point of $I_\lambda$ in the case where $\lambda > \lambda_1$ is not one of the eigenvalues $\lambda_k$. Fix $\lambda'$ such that $\lambda_k < \lambda' < \lambda < \lambda_{k+1}$, and let $\delta > 0$ be so small that the conclusions of Proposition \ref{Proposition 3} hold with $\lambda_k + C \delta^{N-sp} < \lambda'$, in particular,
\begin{equation} \label{41}
\Psi(w) < \lambda' \quad \forall w \in E_\delta.
\end{equation}
Then take $A_0 = E_\delta$ and $B_0 = \Psi_{\lambda_{k+1}}$, and note that $A_0$ and $B_0$ are disjoint nonempty closed symmetric subsets of $\M$ such that
\[
i(A_0) = i(\M \setminus B_0) = k
\]
by Proposition \ref{Proposition 3} and \eqref{5}. Now let $0 < \eps \le \delta/2$, let $R > r > 0$, let $v_0 = \pi(u_{\eps,\delta}) \in \M \setminus E_\delta$, and let $A$, $B$ and $X$ be as in Theorem \ref{Theorem 2}.

For $u \in \Psi_{\lambda_{k+1}}$,
\[
I_\lambda(ru) \ge \frac{1}{p} \left(1 - \frac{\lambda}{\lambda_{k+1}}\right) r^p - \frac{1}{p_s^\ast\, S^{p_s^\ast/p}}\, r^{p_s^\ast}.
\]
Since $\lambda < \lambda_{k+1}$, it follows that $\inf I_\lambda(B) > 0$ if $r$ is sufficiently small.

Next we show that $I_\lambda \le 0$ on $A$ if $R$ is sufficiently large. For $w \in E_\delta$ and $t \ge 0$,
\[
I_\lambda(tw) \le \frac{t^p}{p} \left(1 - \frac{\lambda}{\Psi(w)}\right) \le 0
\]
by \eqref{41}. Now let $w \in E_\delta$ and $0 \le t \le 1$, and set $u = \pi((1 - t)\, w + tv_0)$. Clearly, $\norm{(1 - t)\, w + tv_0} \le 1$, and since the supports of $w$ and $v_0$ are disjoint by Proposition \ref{Proposition 3},
\[
\pnorm[p_s^\ast]{(1 - t)\, w + tv_0}^{p_s^\ast} = (1 - t)^{p_s^\ast} \pnorm[p_s^\ast]{w}^{p_s^\ast} + t^{p_s^\ast} \pnorm[p_s^\ast]{v_0}^{p_s^\ast}.
\]
In view of \eqref{30} and since
\begin{equation} \label{42}
\pnorm[p_s^\ast]{v_0}^{p_s^\ast} = \frac{\pnorm[p_s^\ast]{u_{\eps,\delta}}^{p_s^\ast}}{\norm{u_{\eps,\delta}}^{p_s^\ast}} \ge \frac{1}{S^{N/(N-sp)}} + \O(\eps^{(N-sp)/(p-1)})
\end{equation}
by Lemma \ref{Lemma 3}, it follows that
\[
\pnorm[p_s^\ast]{u}^{p_s^\ast} = \frac{\pnorm[p_s^\ast]{(1 - t)\, w + tv_0}^{p_s^\ast}}{\norm{(1 - t)\, w + tv_0}^{p_s^\ast}} \ge \frac{1}{C}
\]
if $\eps$ is sufficiently small, where $C = C(N,\Omega,p,s,k) > 0$. Then
\[
I_\lambda(Ru) \le \frac{R^p}{p} - \frac{R^{p_s^\ast}}{p_s^\ast} \pnorm[p_s^\ast]{u}^{p_s^\ast} \le \frac{R^p}{p} - \frac{R^{p_s^\ast}}{p_s^\ast\, C} \le 0
\]
if $R$ is sufficiently large.
In view of \eqref{43} and Proposition \ref{Proposition 1}, it only remains to show that 
$$
\sup I_\lambda(X) < \dfrac{s}{N}\, S^{N/sp},
$$ 
if $\eps$ is sufficiently small. Noting that
\[
X = \set{\rho\, \pi((1 - t)\, w + tv_0) : w \in E_\delta,\, 0 \le t \le 1,\, 0 \le \rho \le R},
\]
let $w \in E_\delta$ and $0 \le t \le 1$, and set $u = \pi((1 - t)\, w + tv_0)$. Then
\begin{align} \label{35}
\sup_{0 \le \rho \le R}\, I_\lambda(\rho u) &\le \sup_{\rho \ge 0}\, \left[\frac{\rho^p}{p} \left(1 - \lambda \pnorm{u}^p\right) - \frac{\rho^{p_s^\ast}}{p_s^\ast} \pnorm[p_s^\ast]{u}^{p_s^\ast}\right] = \frac{s}{N} \left[\frac{\left(1 - \lambda \pnorm{u}^p\right)^+}{\pnorm[p_s^\ast]{u}^p}\right]^{N/sp} \\
&= \frac{s}{N} \left[\frac{\left(\norm{(1 - t)\, w + tv_0}^p - \lambda \pnorm{(1 - t)\, w + tv_0}^p\right)^+}{\pnorm[p_s^\ast]{(1 - t)\, w + tv_0}^p}\right]^{N/sp}.   \notag
\end{align}
Since $w = 0$ in $B_{2 \theta \delta}(0)$ by Proposition \ref{Proposition 3} and $v_0 = 0$ in $B_{\theta \delta}(0)^c$ by \eqref{20},
\begin{align} \label{36}
\norm{(1 - t)\, w + tv_0}^p & \le (1 - t)^p \int_{A_1} \frac{|w(x) - w(y)|^p}{|x - y|^{N+sp}}\, dx dy + t^p \int_{A_2} \frac{|v_0(x) - v_0(y)|^p}{|x - y|^{N+sp}}\, dx dy\\[5pt]
&+ 2 \int_{A_3} \frac{|(1 - t)\, w(x) - tv_0(y)|^p}{|x - y|^{N+sp}}\, dx dy =: (1 - t)^p\, I_1 + t^p\, I_2 + 2 I_3,   \notag
\end{align}
where 
$$
A_1 = B_{\theta \delta}(0)^c \times B_{\theta \delta}(0)^c,\quad 
A_2 = B_{2 \theta \delta}(0) \times B_{2 \theta \delta}(0),\quad 
A_3 = B_{2 \theta \delta}(0)^c \times B_{\theta \delta}(0).
$$ 
We estimate $I_3$ using the following elementary inequality: given $\kappa > 1$ and $p - 1 < q < p$, there exists a constant $C = C(\kappa,q) > 0$ such that
\[
|a + b|^p \le \kappa\, |a|^p + |b|^p + C\, |a|^{p-q}\, |b|^q \quad \forall a, b \in \R.
\]
Taking $\kappa = \lambda/\lambda'$ and, thanks to $N>sp^2$, choosing $q\in \ ]N(p - 1)/(N - sp), p[$, we get
\begin{align} \label{38}
I_3 &\le \frac{\lambda}{\lambda'}\, (1 - t)^p \int_{A_3} \frac{|w(x) - w(y)|^p}{|x - y|^{N+sp}}\, dx dy + t^p \int_{A_3} \frac{|v_0(x) - v_0(y)|^p}{|x - y|^{N+sp}}\, dx dy  \\
&+ C \int_{A_3} \frac{|w(x)|^{p-q}\, v_0(y)^q}{|x - y|^{N+sp}}\, dx dy =: \frac{\lambda}{\lambda'}\, (1 - t)^p\, I_4 + t^p\, I_5 + C J_q. \notag
\end{align}
Clearly, $I_1 + 2 I_4 \le \norm{w}^p = 1$ and $I_2 + 2 I_5 \le \norm{v_0}^p = 1$. By \eqref{30} and since 
$$
|x - y| \ge |x| - \theta \delta \ge |x|/2,\quad\text{on $A_3$},
$$
we have
\[
J_q \le \frac{C}{\norm{u_{\eps,\delta}}^q} \int_{A_3} \frac{u_{\eps,\delta}(y)^q}{|x|^{N+sp}}\, dx dy \le \frac{C}{\delta^{sp}} \int_{\R^N} u_{\eps,\delta}(y)^q\, dy
\]
since \eqref{23} implies that $\pnorm[p_s^\ast]{u_{\eps,\delta}}$, and hence also $\norm{u_{\eps,\delta}}$, is bounded away from zero if $\eps$ is sufficiently small. Recalling \eqref{defG}, it holds $G_{\eps,\delta}(t) \le t$ for all $t \ge 0$, and thus
\[
\int_{\R^N} u_{\eps,\delta}(y)^q\, dy \le \int_{\R^N} U_\eps(y)^q\, dy = \eps^{N-(N-sp)\, q/p} \int_{\R^N} U(y)^q\, dy,
\]
and the last integral is finite by \eqref{10} since $q > N(p - 1)/(N - sp)$. So combining \eqref{36} and \eqref{38} gives
\begin{equation}
\label{po}
\norm{(1 - t)\, w + tv_0}^p \le \frac{\lambda}{\lambda'}\, (1 - t)^p + t^p + C\, \eps^{N-(N-sp)\, q/p}.
\end{equation}
On the other hand, since the supports of $w$ and $v_0$ are disjoint,
\begin{align} \label{49}
\pnorm{(1 - t)\, w + tv_0}^p &= (1 - t)^p \pnorm{w}^p + t^p \pnorm{v_0}^p,\\
\pnorm[p_s^\ast]{(1 - t)\, w + tv_0}^{p_s^\ast} &= (1 - t)^{p_s^\ast} \pnorm[p_s^\ast]{w}^{p_s^\ast} + t^{p_s^\ast} \pnorm[p_s^\ast]{v_0}^{p_s^\ast}.
\notag
\end{align}
By \eqref{41}, $\pnorm{w}^p = 1/\Psi(w) > 1/\lambda'$. By \eqref{30}, $\pnorm[p_s^\ast]{w}$ is bounded away from zero, and \eqref{42} implies that so is $\pnorm[p_s^\ast]{v_0}$ if $\eps$ is sufficiently small, so the last expression in \eqref{49} is bounded away from zero. It follows from \eqref{po} and \eqref{49} that
\[
\frac{\norm{(1 - t)\, w + tv_0}^p - \lambda \pnorm{(1 - t)\, w + tv_0}^p}{\pnorm[p_s^\ast]{(1 - t)\, w + tv_0}^p} \le \frac{1 - \lambda \pnorm{v_0}^p}{\pnorm[p_s^\ast]{v_0}^p} + C\, \eps^{N-(N-sp)\, q/p}.
\]
Since $v_0 = u_{\eps,\delta}/\norm{u_{\eps,\delta}}$, the right-hand side is less than or equal to
\[
S_{\eps,\delta}(\lambda) + C\, \eps^{N-(N-sp)\, q/p} \le S - \left(\frac{\lambda}{C} - C\, \eps^{(N-sp^2)/(p-1)} - C\, \eps^{(N-sp)(1-q/p)}\right) \eps^{sp}
\]
by \eqref{40}. Since $N > sp^2$ and $q < p$, it follows from this that the last expression in \eqref{35} is strictly less than $\dfrac{s}{N}\, S^{N/sp}$ if $\eps$ is sufficiently small.

\subsection{Case 3: $N^2/(N + s) > sp^2$ and $\lambda = \lambda_k$}

Let $\lambda = \lambda_k < \lambda_{k+1}$, let $\delta > 0$ be so small that the conclusions of Proposition \ref{Proposition 3} hold with $\lambda_k + C \delta^{N-sp} < \lambda_{k+1}$, in particular, $\Psi(w) < \lambda_{k+1}$ for all $w \in E_\delta$, and take $A_0 = E_\delta$ and $B_0 = \Psi_{\lambda_{k+1}}$ as in the last subsection. Then let $0 < \eps \le \delta/2$, let $R > r > 0$, let $v_0 = \pi(u_{\eps,\delta}) \in \M \setminus E_\delta$, and let $A$, $B$ and $X$ be as in Theorem \ref{Theorem 2}. As before, $\inf I_\lambda(B) > 0$ if $r$ is sufficiently small and
\[
I_\lambda(R\, \pi((1 - t)\, w + tv_0)) \le 0 \quad \forall w \in E_\delta,\, 0 \le t \le 1
\]
if $R$ is sufficiently large. On the other hand,
\[
I_\lambda(tw) \le \frac{t^p}{p} \left(1 - \frac{\lambda_k}{\Psi(w)}\right) \le C R^p \delta^{N-sp} \quad \forall w \in E_\delta,\, 0 \le t \le R
\]
by \eqref{31}, where $C$ denotes a generic positive constant independent of $\eps$ and $\delta$. It follows that
\[
\sup I_\lambda(A) \le C R^p \delta^{N-sp}< \inf I_\lambda(B)
\]
if $\delta$ is sufficiently small. As in the last proof, it only remains to show that (see \eqref{35})
\begin{equation} \label{48}
\sup_{(w,t) \in E_\delta \times [0,1]}\, \frac{\norm{(1 - t)\, w + tv_0}^p - \lambda_k \pnorm{(1 - t)\, w + tv_0}^p}{\pnorm[p_s^\ast]{(1 - t)\, w + tv_0}^p} < S
\end{equation}
if $\eps$ and $\delta$ are suitably small.\
We estimate the integral $I_3$ in \eqref{36} using the elementary inequality
\begin{equation} \label{63}
|a + b|^p \le |a|^p + |b|^p + C \left(|a|^{p-1}\, |b| + |a|\, |b|^{p-1}\right) \quad \forall a, b \in \R
\end{equation}
to get
\begin{align} \label{44}
I_3  & \le (1 - t)^p \int_{A_3} \frac{|w(x) - w(y)|^p}{|x - y|^{N+sp}}\, dx dy + t^p \int_{A_3} \frac{|v_0(x) - v_0(y)|^p}{|x - y|^{N+sp}}\, dx dy\\[5pt]
&+ C\, (1 - t)^{p-1} \int_{A_3} \frac{|w(x)|^{p-1}\, v_0(y)}{|x - y|^{N+sp}}\, dx dy + C\, (1 - t) \int_{A_3} \frac{|w(x)|\, v_0(y)^{p-1}}{|x - y|^{N+sp}}\, dx dy \notag\\
&=: (1 - t)^p\, I_4 + t^p\, I_5 + C\, (1 - t)^{p-1} J_1 + C\, (1 - t) J_{p-1}.   \notag
\end{align}
As before, $I_1 + 2 I_4, I_2 + 2 I_5 \le 1$ and for $q = 1, p - 1$,
\begin{align*}
J_q := &\int_{A_3} \frac{|w(x)|^{p-q}\, v_0(y)^q}{|x - y|^{N+sp}}\, dx dy \le C \int_{A_3} \frac{u_{\eps,\delta}(y)^q}{|x|^{N+sp}}\, dx dy \le \frac{C}{\delta^{sp}} \int_{B_{\theta \delta}(0)} U_\eps(y)^q\, dy \\
\le &\frac{C\, \eps^{N-(N-sp)\, q/p}}{\delta^{sp}} \int_{B_{\theta \delta/\eps}(0)} U(y)^q\, dy.
\end{align*}
We take $\delta = \eps^\alpha$ with $\alpha \in (0,1)$ and use \eqref{10} to estimate the last integral to get 
$$
J_q \le C\, \eps^{(N-sp)[p\, (p-q-1)\, \alpha + q]/p\, (p-1)}.
$$ 
So combining \eqref{36} and \eqref{44} gives
\begin{equation} \label{46}
\norm{(1 - t)\, w + tv_0}^p \le (1 - t)^p + t^p + \widetilde{J}_1 + \widetilde{J}_{p-1},
\end{equation}
where
\[
\widetilde{J}_q := C\, (1 - t)^{p-q}\, J_q \le C\, (1 - t)^{p-q}\, \eps^{(N-sp)[p\, (p-q-1)\, \alpha + q]/p\, (p-1)}.
\]
Young's inequality then gives
\begin{equation} \label{52}
\widetilde{J}_q \le \frac{\kappa}{3}\, (1 - t)^{p_s^\ast} + C\, \eps^{sp + \beta_q(\alpha)} \kappa^{- \gamma_q}
\end{equation}
for any $\kappa > 0$, where
\begin{equation*}
\beta_q(\alpha) = \frac{[N^2 - sp^2\, (N + s)](p - 1)(p - q) - Np\, (N - sp)(p - q - 1)(\alpha_0 - \alpha)}{[(N - sp)\, q + sp^2](p - 1)},
\end{equation*}
and
\begin{equation*}
\alpha_0 = \frac{N - sp^2}{N - sp}, \qquad \gamma_q = \frac{(N - sp)(p - q)}{Np - (N - sp)(p - q)}.
\end{equation*}
Then
\begin{equation} \label{47}
\norm{(1 - t)\, w + tv_0}^p \le (1 - t)^p + t^p + \frac{2 \kappa}{3}\, (1 - t)^{p_s^\ast} + C\, \eps^{sp} \left(\eps^{\beta_1(\alpha)} \kappa^{- \gamma_1} + \eps^{\beta_{p-1}(\alpha)} \kappa^{- \gamma_{p-1}}\right)
\end{equation}
by \eqref{46} and \eqref{52}. 
Using $N^2/(N + s) > sp^2$, we fix $\alpha < \alpha_0$ so close to $\alpha_0$ that $\beta_q(\alpha) > 0$ for $q = 0, 1, p - 1, p$. By \eqref{31} and Young's inequality,
\begin{equation} \label{53}
\lambda_k\, (1 - t)^p \pnorm{w}^p \ge (1 - t)^p \left(1 - C\, \eps^{(N-sp)\, \alpha}\right) \ge (1 - t)^p - \frac{\kappa}{3}\, (1 - t)^{p_s^\ast} - C\, \eps^{sp + \beta_0(\alpha)} \kappa^{- \gamma_0}.
\end{equation}
By \eqref{47}, \eqref{49}, and \eqref{53}, the quotient $Q(w,t)$ in \eqref{48} satisfies
\begin{equation} \label{54}
Q(w,t) \le \frac{\big(1 - \lambda_k \pnorm{v_0}^p\big)\, t^p + \kappa\, (1 - t)^{p_s^\ast} + C\, \eps^{sp + \beta(\alpha)} \kappa^{- \gamma}}{\left[(1 - t)^{p_s^\ast} \pnorm[p_s^\ast]{w}^{p_s^\ast} + t^{p_s^\ast} \pnorm[p_s^\ast]{v_0}^{p_s^\ast}\right]^{p/p_s^\ast}},
\end{equation}
where
\[
\beta(\alpha) = \min \set{\beta_0(\alpha),\beta_1(\alpha),\beta_{p-1}(\alpha)} > 0, \qquad \gamma = \max \set{\gamma_0,\gamma_1,\gamma_{p-1}} = \frac{N}{sp} - 1.
\]
As before, the denominator is bounded away from zero if $\eps$ is sufficiently small, so it follows that
\[
\sup_{(w,t) \in E_{\eps^\alpha} \times [0,t_0)}\, Q(w,t) \leq C(t_0^p+\kappa+\eps^{sp+\beta(\alpha)}\kappa^{-\gamma})< S
\]
for some $t_0 > 0$ if $\kappa$ and $\eps$ are sufficiently small. For $t \ge t_0$, rewriting the right-hand side of \eqref{54} as
\[
\frac{\dfrac{1 - \lambda_k \pnorm{v_0}^p}{\pnorm[p_s^\ast]{v_0}^p} + \dfrac{\kappa\, (1 - t)^{p_s^\ast} + C\, \eps^{sp + \beta(\alpha)} \kappa^{- \gamma}}{t^p \pnorm[p_s^\ast]{v_0}^p}}{\left[\dfrac{\pnorm[p_s^\ast]{w}^{p_s^\ast}}{t^{p_s^\ast} \pnorm[p_s^\ast]{v_0}^{p_s^\ast}}\, (1 - t)^{p_s^\ast} + 1\right]^{p/p_s^\ast}}
\]
gives $Q(w,t) \le g((1 - t)^{p_s^\ast})$, where
\[
g(\tau) = \frac{S_{\eps,\eps^\alpha}(\lambda_k) + C \left(\kappa \tau + \eps^{sp + \beta(\alpha)} \kappa^{- \gamma}\right)}{(1 + C^{-1}\, \tau)^{p/p_s^\ast}}, \qquad C=C(N, p, s, t_0).
\]
Since $0 \le (1 - t)^{p_s^\ast} < 1$, then
\[
Q(w,t) \le S_{\eps,\eps^\alpha}(\lambda_k) + C \big(\kappa + \eps^{sp + \beta(\alpha)} \kappa^{- \gamma}\big).
\]
If $S_{\eps_j,\eps_j^\alpha}(\lambda_k) < S/2$ for some sequence $\eps_j \to 0$, then the right-hand side is less than $S$ for sufficiently small $\kappa$ and $\eps = \eps_j$ with sufficiently large $j$, so we may assume that $S_{\eps,\eps^\alpha}(\lambda_k) \ge S/2$ for all sufficiently small $\eps$. Then it is easily seen that if $\kappa \le (p/p_s^\ast)\, S/2\, (C + 1)$, then $g'(\tau) \le 0$ for all $\tau \in [0,1]$ and hence the maximum of $g((1 - t)^{p_s^\ast})$ on $[t_0,1]$ occurs at $t = 1$. So, we reach 
\[
Q(w,t) \le S_{\eps,\eps^\alpha}(\lambda_k) + C\, \eps^{sp + \beta(\alpha)} \kappa^{- \gamma} \le S - \left(\frac{\lambda_k}{C} - C\, \eps^{\beta_p(\alpha)} - C\, \eps^{\beta(\alpha)} \kappa^{- \gamma}\right) \eps^{sp}
\]
by \eqref{40}, and the desired conclusion follows for sufficiently small $\kappa$ and $\eps$.

\subsection{Case 4: $(N^3 + s^3 p^3)/N\, (N + s) > sp^2$, $\bdry{\Omega} \in C^{1,1}$, and $\lambda = \lambda_k$}

By Iannizzotto et al.\! \cite[Theorem 4.4]{IaMoSq}, there exists a constant $C = C(N,\Omega,p,s) > 0$ such that for any $v \in W^{s,p}_0(\Omega)$ with $(- \Delta)_p^s\, v \in L^\infty(\Omega)$,
\begin{equation} \label{55}
|v(x)| \le C \pnorm[\infty]{(- \Delta)_p^s\, v}^{1/(p-1)} d^s(x) \quad \forall x \in \R^N,
\end{equation}
where $d(x) = \dist{x}{\R^N \setminus \Omega}$.

\begin{lemma} \label{Lemma 5}
Assume that $\bdry{\Omega} \in C^{1,1}$. Given $\alpha, \beta > 1$, there exists a constant $C = C(N,\Omega,p,s,\alpha,\beta) > 0$ such that if $B_{\beta r}(0) \subset \set{x \in \Omega : d(x) < \alpha r}$, then for any $v \in W^{s,p}_0(\Omega)$ with $(- \Delta)_p^s\, v \in L^\infty(\Omega)$,
\[
|v(x) - v(y)| \le C \pnorm[\infty]{(- \Delta)_p^s\, v}^{1/(p-1)} |x - y|^s \quad \forall x \in B_r(0),\,  y \in \Omega \setminus B_{\beta r}(0).
\]
\end{lemma}

\begin{proof}
By \eqref{55},
\[
|v(x) - v(y)| \le C \pnorm[\infty]{(- \Delta)_p^s\, v}^{1/(p-1)} (d^s(x) + d^s(y)).
\]
Since $d(x) \le \alpha r$ and $|x - y| \ge (\beta - 1)\, r$,
\[
d(x) \le \frac{\alpha}{\beta - 1}\, |x - y|,
\]
and since $d(y) \le d(x) + |x - y|$ by the triangle inequality and $s < 1$,
\[
d^s(y) \le d^s(x) + |x - y|^s.
\]
So the desired inequality holds with the constant $C\, (2 \alpha^s/(\beta - 1)^s + 1)$.
\end{proof}

\noindent
Let $\eta_\delta$ be the cut-off function in Lemma \ref{Lemma 6}.

\begin{lemma} \label{Lemma 8}
Assume that $\bdry{\Omega} \in C^{1,1}$. Then there exists a constant $C = C(N,\Omega,p,s) > 0$ such that for any $v \in W^{s,p}_0(\Omega)$ such that $(- \Delta)_p^s\, v \in L^\infty(\Omega)$ and $\delta > 0$ such that $B_{6 \theta \delta}(0) \subset \set{x \in \Omega : d(x) < 12 \theta \delta}$,
\begin{equation} \label{56}
\norm{v \eta_\delta}^p \le \norm{v}^p + C \pnorm[\infty]{(- \Delta)_p^s\, v}^{p/(p-1)} \delta^N.
\end{equation}
\end{lemma}

\begin{proof}
Set $f = (- \Delta)_p^s\, v$ and $K = \pnorm[\infty]{f} < \infty$. Then
\begin{equation} \label{58}
|v(x)| \le CK^{1/(p-1)}\, \delta^s \quad \forall x \in B_{6 \theta \delta}(0)
\end{equation}
by \eqref{55}, and for $k = 3, 5$,
\begin{equation} \label{57}
|v(x) - v(y)| \le CK^{1/(p-1)}\, |x - y|^s \quad \forall x \in B_{k \theta \delta}(0),\, y \in \Omega \setminus B_{(k+1) \theta \delta}(0)
\end{equation}
by Lemma \ref{Lemma 5}. We proceed splitting $\norm{v \eta_\delta}^p$ as in the proof of Lemma \ref{Lemma 6}, and estimate the integral
\[
I_3 = \int_{A_3} \frac{|v(x) - v(y) + v(x)\, (\eta_\delta(x) - 1)|^p}{|x - y|^{N+sp}}\, dx dy
\]
in \eqref{59} using the elementary inequality
\[
|a + b|^p \le |a|^p + C \left(|a|^{p-1}\, |b| + |b|^p\right) \quad \forall a, b \in \R
\]
to get
\[
\begin{split}
I_3 &\le \int_{A_3} \frac{|v(x) - v(y)|^p}{|x - y|^{N+sp}}\, dx dy + C \int_{A_3} \frac{|v(x) - v(y)|^{p-1}\, |v(x)|}{|x - y|^{N+sp}}\, dx dy\\
&\qquad+ C \int_{A_3} \frac{|v(x)|^p}{|x - y|^{N+sp}}\, dx dy =: I_4 + C I_5 + C I_6.
\end{split}
\]
We have $I_1 + 2 I_4 \le \norm{v}^p$. By \eqref{58} and \eqref{57}, and since $|x - y| \ge |y|/4$ on $A_3$,
\begin{align*}
I_5  &\le CK^{p/(p-1)}\, \delta^s \int_{A_3} \frac{dx dy}{|y|^{N+s}} = CK^{p/(p-1)}\, \delta^N, \\
I_6  &\le CK^{p/(p-1)}\, \delta^{sp} \int_{A_3} \frac{dx dy}{|y|^{N+sp}} = CK^{p/(p-1)}\, \delta^N.
\end{align*}
To estimate $I_2$, let $\varphi_\delta$ be as in the proof of Lemma \ref{Lemma 6}. Since $\varphi_\delta = \eta_\delta$ in $A_2$,
\[
I_2 \le C \left(\int_{A_2} \frac{|v(x)|^p\, |\varphi_\delta(x) - \varphi_\delta(y)|^p}{|x - y|^{N+sp}}\, dx dy + \int_{A_2} \frac{|v(x) - v(y)|^p\, \varphi_\delta(y)^p}{|x - y|^{N+sp}}\, dx dy\right) =: C\, (I_7 + I_8).
\]
By \eqref{58} and $\norm{\varphi_\delta}^p=\delta^{N-ps}\norm{\varphi}^p$ by scaling, we get
\begin{equation} \label{61}
I_7 \le CK^{p/(p-1)}\, \delta^{sp} \norm{\varphi_\delta}^p = CK^{p/(p-1)}\, \delta^N.
\end{equation}
By Lemma \ref{Lemma 7},
\[
I_8 \le 2 \int_\Omega f(x)\, v(x)\, \varphi_\delta(x)^p\, dx + C \int_{\R^{2N}} \frac{|v(x)|^p\, |\varphi_\delta(x) - \varphi_\delta(y)|^p}{|x - y|^{N+sp}}\, dx dy =: 2 I_9 + C I_{10}.
\]
Since $\varphi = 0$ outside $B_{5 \theta \delta}(0)$,
\[
I_9 \le \int_{B_{5 \theta \delta}(0)} |f(x)|\, |v(x)|\, dx \le CK^{p/(p-1)}\, \delta^{N+s}
\]
by \eqref{58} again. Changing variables gives
\[
I_{10} = \delta^{N-sp} \int_{\R^{2N}} \frac{|v(\delta x)|^p\, |\varphi(x) - \varphi(y)|^p}{|x - y|^{N+sp}}\, dx dy.
\]
We have $|v(\delta x)| \le CK^{1/(p-1)}\, d^s(\delta x)$ by \eqref{55}, and $d(\delta x) \le d(0) + \delta\, |x| \le C \delta$ since $d(0) \le 6 \theta \delta$ and $\Omega$ is bounded, so the last integral is less than or equal to $CK^{p/(p-1)}\, \delta^{sp} \norm{\varphi}^p$. Hence $I_{10} \le CK^{p/(p-1)}\, \delta^N$.
\end{proof}

\noindent
Since $\bdry{\Omega} \in C^{1,1}$, for all sufficiently small $\delta > 0$, the ball $B_{6 \theta \delta}(0)$ is contained in $\set{x \in \Omega : d(x) < 12 \theta \delta}$ after a translation. Then by Lemma \ref{Lemma 8} and Proposition \ref{Proposition 2},
\[
\norm{v_\delta}^p \le 1 + C \delta^N \quad \forall v \in E,
\]
and using this inequality in place of \eqref{32} in the proof of Proposition \ref{Proposition 3} shows that \eqref{31} can now be strengthened to
\begin{equation} \label{64}
\sup_{w \in E_\delta}\, \Psi(w) \le \lambda_k + C \delta^N.
\end{equation}
Proceeding as in the last subsection, we have to verify \eqref{48} for suitably small $\eps$ and $\delta$.
Since the argument is similar, we only point out where it differs. Let $v \in E$ and let 
$$
w = \pi(v_\delta) = v_\delta/\norm{v_\delta}. 
$$
As noted in the proof of Proposition \ref{Proposition 3}, $\norm{v_\delta}$ is bounded away from zero, so
\[
J_q \le C \int_{A_3} \frac{|v_\delta(x)|^{p-q}\, u_{\eps,\delta}(y)^q}{|x - y|^{N+sp}}\, dx dy,
\]
where $A_3 = B_{2 \theta \delta}(0)^c \times B_{\theta \delta}(0)$. By Lemma \ref{Lemma 5}, \eqref{55}, and Proposition \ref{Proposition 2}, and since $$
|x - y| \ge |x|/2 \ge \theta \delta, \quad\text{on $A_3$},
$$
we get
\begin{align*}
|v_\delta(x)|^{p-q}  &\le |v(x)|^{p-q} \le C \left(|v(x) - v(y)|^{p-q} + |v(y)|^{p-q}\right)  \\
&\le C\, \big(|x - y|^{s(p-q)} + \delta^{s(p-q)}\big)
\le C\, |x - y|^{s(p-q)},
\end{align*}
so
\[
J_q \le C \int_{A_3} \frac{u_{\eps,\delta}(y)^q}{|x|^{N+sq}}\, dx dy \le \frac{C}{\delta^{sq}} \int_{B_{\theta \delta}(0)} U_\eps(y)^q\, dy \le C\, \eps^{\{p\, [(p-q-1)\, N + sq]\, \alpha + (N-sp)\, q\}/p\, (p-1)}.
\]
Then \eqref{52} holds with
\[
\beta_q(\alpha) = \frac{[N^3 + s^3 p^3 - sp^2 N\, (N + s)](p - 1)(p - q) - Np\, (N - sp)[N\, (p - q - 1) + sq](\alpha_0 - \alpha)}{(N - sp)[(N - sp)\, q + sp^2](p - 1)},
\]
and so does \eqref{53} by \eqref{64}. Using 
$$
(N^3 + s^3 p^3)/N\, (N + s) > sp^2, 
$$
we fix $\alpha < \alpha_0$ so close to $\alpha_0$ that $\beta_q(\alpha) > 0$ for $q = 0, 1, p - 1, p$ and proceed as before.

\def\cdprime{$''$}

\end{document}